\documentclass[11pt]{article}

\usepackage{amsmath, amsthm, amsopn, amssymb, enumerate}

\newtheorem{thm}{Theorem}[section]
\newtheorem{lemma}[thm]{Lemma}
\newtheorem{cor}[thm]{Corollary}
\newtheorem{claim}{Claim}

\theoremstyle{definition}
\newtheorem{dfn}[thm]{Definition}
\newtheorem{remark}[thm]{Remark}

\newcommand{\Bin}{\mathrm{Bin}}
\newcommand{\Ex}{\mathbb{E}}
\newcommand{\ex}{\mathrm{ex}}

\newcommand{\A}{\mathcal{A}}
\newcommand{\bU}{\hat{U}}
\newcommand{\B}{\mathcal{B}}
\newcommand{\bH}{\mathbf{H}}
\newcommand{\bB}{\mathfrak{B}}
\newcommand{\bG}{\mathbf{G}}
\newcommand{\bomega}{\mathbf{\omega}}

\newcommand{\bp}{\mathbf{p}}
\renewcommand{\P}{\mathcal{P}}
\newcommand{\N}{\mathbb{N}}
\newcommand{\eps}{\varepsilon}
\renewcommand{\S}{\mathcal{S}}
\newcommand{\bUU}{\hat{\mathcal{U}}}

\author{Wojciech Samotij\thanks{Trinity College, Cambridge CB2 1TQ, UK. E-mail address: {\tt ws299@cam.ac.uk}. Research supported in part by a~Trinity College JRF.}}
\title{Stability results for random discrete structures}
\date{\today}

\begin{document}

\maketitle

\begin{abstract}
  Two years ago, Conlon and Gowers, and Schacht proved general theorems that allow one to~transfer a large class of extremal combinatorial results from the~deterministic to the~probabilistic setting. Even though the~two papers solve the~same set of long-standing open problems in~probabilistic combinatorics, the~methods used in them vary significantly and therefore yield results that are not comparable in certain aspects. the~theorem of Schacht can be applied in~a~more general setting and yields stronger probability estimates, whereas the~one of Conlon and Gowers also implies random versions of some structural statements such as the~famous stability theorem of Erd{\H o}s and Simonovits. In this paper, we bridge the~gap between these two transference theorems. Building on the~approach of Schacht, we prove a general theorem that allows one to transfer deterministic stability results to the~probabilistic setting that is somewhat more general and stronger than the~one obtained by Conlon and Gowers. We then use this theorem to derive several new results, among them a random version of the~Erd{\H o}s-Simonovits stability theorem for~arbitrary graphs. the~main new idea, a refined approach to multiple exposure when considering subsets of binomial random sets, may be of independent interest.
\end{abstract}

\section{Introduction}

One of the~most active areas of research within combinatorics has always been the~study of various \emph{extremal problems}. In the~most classical sense, extremal results in combinatorics give answers to~questions of the~following general form: For a finite set $X$, what is the~largest subset of $X$ that does not contain subsets of a particular type? Two archetypal examples of such results are the~famous theorem of Tur{\'a}n~\cite{Tu}, which determines the~maximum number of edges in an~$n$-vertex graph that does not contain a complete subgraph on $k$ vertices, and the~celebrated theorem of~Szemer{\'e}di~\cite{Sz}, which proves that for every positive $\delta$, every subset $A$ of $\{1, \ldots, n\}$ that satisfies $|A| \geq \delta n$ contains a $k$-term arithmetic progression, provided that $n$ is sufficiently large (as a function of $k$ and $\delta$).

Extremal results are often accompanied by their structural refinements. Among them, the~most notable are various \emph{stability results}, which have the~following general form: Suppose that a subset $Y \subseteq X$ does not contain subsets of a particular type and, moreover, the~number of elements in $Y$ is maximum possible (or close to maximum possible) among all subsets of $X$ with this property (of~not containing subsets of some type). Then $Y$ is very structured. Here, Tur{\'a}n's theorem can again serve as an~example as it not only determines the~maximum number of edges in an~$n$-vertex graph with no $k$-vertex complete subgraph, but also shows that (up to isomorphism) the~only $K_k$-free $n$-vertex graph with this many edges is the~complete $(k-1)$-partite graph with partite sets of equal or nearly equal size, denoted by $T_{k-1}(n)$ and referred to as the~Tur{\'a}n graph. The~stability statement of the~type we will be considering was only proved much later by Erd{\H o}s and Simonovits~\cite{Si}. It states that in fact every $K_k$-free $n$-vertex graph whose number of~edges is `close' to the~number of~edges in $T_{k-1}(n)$ must be very `close' to the~graph $T_{k-1}(n)$, see Theorem~\ref{thm:stability}.

A dominant trend in probabilistic combinatorics in the~past two decades has been the~formulation and study of various `sparse random' analogues of classical extremal problems such as the~aforementioned theorems of Tur{\'a}n and Szemer{\'e}di. Usually, these problems are studied in the~\emph{binomial random model}. For a finite set $X$ and a real number $p \in [0,1]$, we denote by $X_p$ the~\emph{$p$-random subset of $X$}, that is, the~random subset of $X$ such that each element of $X$ belongs to $X_p$ with probability $p$, independently of all other elements. A sparse random analogue of the~theorem of~Szemer{\'e}di is the~assertion that with probability close to $1$, every subset $A$ of $\{1, \ldots, n\}_p$ that satisfies $|A| \geq \delta np$ contains a $k$-term arithmetic progression, provided that $p$ is sufficiently large as~a~function of $n$, $k$, and $\delta$; note that $np$ is the~expected size of the~random set $\{1, \ldots, n\}_p$.

Various problems of this type, in particular the~sparse random version of~Szemer{\'e}di's theorem (Theorem~\ref{thm:Szemeredi-np}), have attracted a tremendous amount of attention from many leading researchers. the~main goal has been to find the~smallest sequence of probabilities $(p_n)$ such that the~statements as the~one above hold \emph{asymptotically almost surely} (a.a.s.~for short), that is, with probability tending to $1$ as $n$, the~size of the~considered structure, tends to infinity. There have been many results in~various special cases, but the~most important general questions, most notably the~random version of Tur{\'a}n's theorem known as the~Haxell-Kohayakawa-{\L}uczak conjecture~\cite{HaKoLu} (or the~Kohayakawa-{\L}uczak-R{\"o}dl conjecture~\cite{KoLuRo}) had remained open until very recently, when all those efforts culminated in two breakthrough results of Conlon and Gowers~\cite{CoGo} and Schacht~\cite{Sc}, which provided a very general and powerful framework to handle problems of this type. The~random versions of~Szemer{\'e}di's and Tur{\'a}n's theorems followed as simple corollaries.

Following~\cite{CoGo}, let us say that a set $A$ of integers is \emph{$(\delta,k)$-Szemer{\'e}di} if every subset of $A$ of~cardinality at least $\delta|A|$ contains a $k$-term arithmetic progression. Also, let us abbreviate $\{1, \ldots, n\}$ by $[n]$. the~methods of Conlon and Gowers, and Schacht imply that in the~$p$-random subset of $[n]$, the~property of being $(\delta,k)$-Szemer{\'e}di has a threshold at $n^{-1/(k-1)}$.

\begin{thm}[{\cite{CoGo,Sc}}]
  \label{thm:Szemeredi-np}
  For every positive $\delta$ and every integer $k$ with $k \geq 3$, there exist positive constants $c$ and $C$ such that
  \[
  \lim_{n \to \infty} P(\text{$[n]_{p_n}$ is $(\delta,k)$-Szemer{\'e}di}) =
  \begin{cases}
    1, & \text{if $p_n \geq Cn^{-\frac{1}{k-1}}$}, \\
    0, & \text{if $p_n \leq cn^{-\frac{1}{k-1}}$}.
  \end{cases}
  \]
\end{thm}

Given two graphs $G$ and $H$, let $\ex(G,H)$ denote the~maximum number of edges in a subgraph of $G$ that is \emph{$H$-free}, that is, does not contain $H$ as a subgraph, i.e.,
\[
\ex(G,H) = \max \{ e(G') \colon G' \subseteq G \text{ and } G' \nsupseteq H \}.
\]
The aforementioned theorem of Tur{\'a}n determines $\ex(K_n,K_k)$ for all $k$ and $n$. It was later generalised by Erd{\H o}s and Stone~\cite{ErSt}, and Erd{\H o}s and Simonovits~\cite{ErSi}, who proved that for an~arbitrary graph $H$ with at least one edge,
\begin{equation}
  \label{eq:ErSt}
  \ex(K_n,H) \leq \left(1 - \frac{1}{\chi(H) - 1} + o(1)\right) \binom{n}{2},
\end{equation}
where $\chi(H)$ is the~chromatic number of $H$. On the~other hand, it is easy to see that for every graph $G$,
\begin{equation}
  \label{eq:exGH}
  \ex(G,H) \geq \left(1 - \frac{1}{\chi(H) - 1}\right) e(G).
\end{equation}
Erd{\H o}s and Simonovits~\cite{Si} proved the~following structural refinement of~\eqref{eq:ErSt}, known since under the~name of Erd{\H o}s-Simonovits stability theorem:
\begin{thm}
  \label{thm:stability}
  For every positive $\delta$ and every graph $H$ with at least one edge, there exists a~positive $\eps$ such that every $n$-vertex $H$-free graph with at least $\ex(K_n,H) - \eps n^2$ edges can be made $(\chi(H) - 1)$-partite by removing from it at most $\delta n^2$ edges.
\end{thm}

Let $G(n,p)$ denote the~binomial random graph on the~vertex set $[n]$ with edge probability $p$ and note that in our notation, $G(n,p) = E(K_n)_p$. A~notion that is intrinsic to the~study of subgraphs of random graphs is that of \emph{$2$-density}. Let $H$ be a~graph with at least $3$ vertices. We define the~$2$-density of $H$, denoted by $m_2(H)$, by
\[
m_2(H) = \max\left\{ \frac{e(K) - 1}{v(K) - 2} \colon K \subseteq H \text{ with } v(K) \geq 3\right\},
\]
where $v(K)$ and $e(K)$ denote the~number of vertices and the~number of edges of $K$, respectively. A~fairly straightforward computation (see, for example, \cite{Sc}) shows that for every $H$, if $p_n \ll n^{-1/m_2(H)}$, then a.a.s.~the number of copies of some $H' \subseteq H$ in $G(n,p)$ is much smaller than $\binom{n}{2}p$, the~expected number of edges in $G(n,p)$, and therefore $\ex(G(n,p),H) = (1+o(1))\binom{n}{2}p$, which is very far from~\eqref{eq:exGH}. Haxell, Kohayakawa, and {\L}uczak~\cite{HaKoLu} conjectured that once $p_n \geq C_H n^{-1/m_2(H)}$, then the~trivial estimate~\eqref{eq:exGH} becomes essentially best possible and hence a~natural generalisation of~\eqref{eq:ErSt} holds in $G(n,p)$. Their conjecture was confirmed by Conlon and Gowers, and Schacht.
\begin{thm}[{\cite{CoGo,Sc}}]
  \label{thm:Turan-Gnp}
  For every graph $H$ with at least one vertex contained in at least two edges and every positive $\eps$, there exist positive constants $c$ and $C$ such that
  \[
  \lim_{n \to \infty} P\left( \ex(G(n,p_n), H) \leq \left(1 - \frac{1}{\chi(H) - 1} + \eps\right)\binom{n}{2}p_n\right) =
  \begin{cases}
    1, & \text{if $p_n \geq Cn^{-1/m_2(H)}$},\\
    0, & \text{if $p_n \leq cn^{-1/m_2(H)}$}.
  \end{cases}
  \]
\end{thm}
Theorem~\ref{thm:Turan-Gnp} showed a~certain advantage of the~approach of Schacht~\cite{Sc} over the~methods of Conlon and Gowers~\cite{CoGo}, which allowed to prove the~above statement only in the~case when $H$ is \emph{strictly $2$-balanced}\footnote{Actually, the~methods of~\cite{CoGo} allow to prove the~$1$-statement in Theorem~\ref{thm:Turan-Gnp} and Theorem~\ref{thm:stability-Gnp} also in the~case when $H$ is only $2$-balanced, i.e., if $m_2(H') \leq m_2(H)$ for all $H' \subseteq H$, under the~somewhat stronger assumption that $p_n \geq n^{-1/m_2(H)}(\log n)^C$ for some constant $C$.}; a~graph $H$ is strictly $2$-balanced if it has the~largest $2$-density among all of its subgraphs or, in other words, if every proper subgraph $H' \subsetneq H$ satisfies $m_2(H') < m_2(H)$. Moreover, Schacht's approach yielded an~asymptotically best possible estimate on the~rate of convergence in the~above limit, showing that the~`error probability' is $\exp(-\Omega(n^2p_n))$, whereas the~other approach yields only an~$n^{-\omega(1)}$ bound. On the~other hand, Conlon and Gowers were able to prove the~following sparse random analogue of the~Erd{\H o}s-Simonovits stability theorem (Theorem \ref{thm:stability}), which did not follow from Schacht's general theorem.

\begin{thm}[{\cite{CoGo}}]
  \label{thm:stability-Gnp}
  For every strictly $2$-balanced graph $H$ and every positive $\delta$, there exist positive constants $C$ and $\eps$ such that if $p_n \geq Cn^{-1/m_2(H)}$, then a.a.s.~every $H$-free subgraph of $G(n,p_n)$ with at least $\left(1 - \frac{1}{\chi(H)-1} -\eps\right)\binom{n}{2}p_n$ edges may be made $(\chi(H)-1)$-partite by removing from it at most $\delta n^2p_n$ edges.
\end{thm}

In view of these disparities, it is natural to ask whether some synthesis of the~methods of~\cite{CoGo} and~\cite{Sc} can bridge the~gap between their results; that is, whether one can prove a~transference theorem that is applicable in all the~settings in which Schacht's result can be applied, gives exponentially decaying bounds on the~`error probability', and yet implies structural statements. In this paper, we give an~affirmative answer to this question. We show how the~approach of Schacht can be adapted to yield structural results of the~form of Theorem~\ref{thm:stability-Gnp} in the~cases where the~methods of Conlon and Gowers are not applicable. We prove a~version of the~general transference theorem from~\cite{Sc} tailored for stability statements. As corollaries of this general theorem, we then derive several new results. In particular, we remove the~assumption that $H$ is (strictly) $2$-balanced from the~statement of Theorem~\ref{thm:stability-Gnp}, where we also improve the~implicit probability estimate from $n^{-\omega(1)}$ to $\exp(-\Omega(n^2p_n))$, which is asymptotically best possible. Finally, we remark that our approach removes the~somewhat artificial condition $p_n \ll 1$ present in the~general transference theorems of both Conlon and Gowers~\cite{CoGo}, and Schacht~\cite{Sc} (the case $p_n = \Omega(1)$ in Theorems~\ref{thm:Turan-Gnp} and~\ref{thm:stability-Gnp} did not follow directly from the~respective transference theorem and required additional arguments). We postpone the~formulation of our main result, Theorem~\ref{thm:main}, to Section~\ref{sec:main-results} and first discuss several of its most important corollaries.

\subsection{New results}

\label{sec:new-results}

In this section, we give a~brief overview of the~applications of our main result, Theorem~\ref{thm:main}. The~proofs of these statements are given in Section~\ref{sec:proofs-new-results}.

\subsubsection{Graphs}

\label{sec:new-graphs}

Our first result generalizes and strengthens Theorem~\ref{thm:stability-Gnp} by removing the~assumption that $H$ is (strictly) $2$-balanced and improving the~probability estimate implicit in the~``asymptotically almost surely'' statement. Theorem~\ref{thm:stability-Gnp-new}, conjectured by Kohayakawa, {\L}uczak, and R{\"o}dl~\cite{KoLuRo}, is an~essentially best possible random analogue of the~stability theorem of Erd{\H o}s and Simonovits (Theorem~\ref{thm:stability}).

\begin{thm}
  \label{thm:stability-Gnp-new}
  For every graph $H$ with at least one vertex contained in at least two edges and every positive $\delta$, there exist positive constants $C$ and $\eps$ such that if $p_n \geq Cn^{-1/m_2(H)}$, then with probability at least $1 - \exp(-\Omega(n^2p_n))$, every $H$-free subgraph of $G(n,p_n)$ with at least $\left(1 - \frac{1}{\chi(H)-1} -\eps\right)\binom{n}{2}p_n$ edges may be made $(\chi(H)-1)$-partite by removing from it at most $\delta n^2p_n$ edges.
\end{thm}

\subsubsection{Hypergraphs}

\label{sec:new-hypergraphs}

Given two $\ell$-uniform hypergraphs $G$ and $H$, similarly as in the~graph case ($\ell = 2$), we define $\ex(G,H)$ to be the~maximum number of edges in an~$H$-free subhypergraph of $G$. Unlike the~graph case, if $\ell \geq 3$, then even the~asymptotic behaviour of the~function $\ex(K_n^{\ell}, H)$ is not known apart from some very specific hypergraphs $H$. Still, for an~arbitrary $H$, it makes sense to define the~\emph{Tur{\'a}n density} of $H$, denoted $\pi(H)$, by
\[
\pi(H) = \lim_{n \to \infty} \frac{\ex(K_n^{(\ell)},H)}{\binom{n}{\ell}},
\]
as a~standard averaging argument shows that the~above limit always exists and that $\pi(H) < 1$ for~every $H$. Moreover, it is not very hard to see that for every hypergraph $G$,
\begin{equation}
  \label{eq:exGH-hyper}
  \ex(G,H) \geq \pi(H)e(G).
\end{equation}

Let $G^{(\ell)}(n,p)$ denote the~binomial random $\ell$-uniform hypergraph on the~vertex set $[n]$ with edge probability $p$. Let $H$ be an~$\ell$-uniform hypergraph with at least $\ell+1$ vertices. Similarly as in the~graph case, we define the~\emph{$\ell$-density} of $H$, denoted by $m_\ell(H)$, by
\[
m_\ell(H) = \max\left\{ \frac{e(K) - 1}{v(K) - \ell} \colon K \subseteq H \text{ with } v(K) \geq \ell+1\right\}.
\]
As in the~case $\ell = 2$, one can see that if $p_n \ll n^{-1/m_\ell(H)}$, then $\ex(G^{\ell}(n,p),H) = (1+o(1))\binom{n}{\ell}p$, which is very far from~\eqref{eq:exGH-hyper}. A~natural generalisation of the~conjecture of Haxell, Kohayakawa, and~{\L}uczak~\cite{HaKoLu} would state that once $p_n \geq C_H n^{-1/m_k(H)}$, then the~trivial estimate~\eqref{eq:exGH-hyper} is essentially best possible. Such statement was proved by Conlon and Gowers, and Schacht.
\begin{thm}[{\cite{CoGo,Sc}}]
  \label{thm:hypergraph-Turan-Gnp}
  For every $\ell$-uniform hypergraph $H$ with at least one vertex contained in~at~least two edges and every positive $\eps$, there exist positive constants $c$ and $C$ such that
  \[
  \lim_{n \to \infty} P\left( \ex(G^{\ell}(n,p_n), H) \leq (\pi(H) + \eps)\binom{n}{\ell}p_n\right) =
  \begin{cases}
    1, & \text{if $p_n \geq Cn^{-1/m_\ell(H)}$},\\
    0, & \text{if $p_n \leq cn^{-1/m_\ell(H)}$}.
  \end{cases}
  \]
\end{thm}
Similarly as in Theorem~\ref{thm:Turan-Gnp}, the~methods of Conlon and Gowers allowed to prove the~above statement only in the~case when $H$ is \emph{strictly $\ell$-balanced}\footnote{Or when $H$ is just $\ell$-balanced, i.e., if $m_\ell(H') \leq m_\ell(H)$ for all $H' \subseteq H$, under the~somewhat stronger assumption that $p_n \geq n^{1-1/m_\ell(H)}(\log n)^C$ in the~$1$-statement.}, i.e., if it has the~largest $\ell$-density among all of its subhypergraphs or, in other words, if every proper subhypergraph $H' \subsetneq H$ satisfies $m_\ell(H') < m_\ell(H)$.

The techniques of Conlon and Gowers can also be used to transfer stability theorems for (strictly) $\ell$-balanced $\ell$-uniform hypergraphs into the~sparse random setting. Unfortunately, unlike the~graph case, where we have the~very general theorem of Erd{\H o}s and Simonovits (Theorem~\ref{thm:stability}), there is only a~handful of stability results known for $\ell$-uniform hypergraphs with $\ell \geq 3$. To the~best of our knowledge, the~only hypergraphs for which an~`Erd{\H o}s-Simonovits-type' stability result is known are: the~Fano plane (the $3$-uniform hypergraph with $7$ vertices and $7$ edges defined by the~points and lines of the~finite projective plane of order $2$), proved independently by Keevash and Sudakov~\cite{KeSu} and F{\"u}redi and Simonovits~\cite{FuSi}; the~$3$-book of $2$ pages (the $3$-uniform hypergraph on the~vertex set $\{1, \ldots, 5\}$ with edge set $\{123, 124, 345\}$), proved by Keevash and Mubayi~\cite{KeMu}; and the~$4$-book of $3$ pages (the $4$-uniform hypergraph on the~vertex set $\{1, \ldots, 7\}$ with edge set $\{1234, 1235, 1236, 4567\}$), proved by F{\"u}redi, Pikhurko, and Simonovits~\cite{FuPiSi}. Among these three hypergraphs, only the~Fano plane is strictly balanced and therefore, the~following result follows from the~methods of Conlon and Gowers.

\begin{thm}[{\cite{CoGo}}]
  \label{thm:Fano-stability}
  For every positive $\delta$, there exist positive constants $C$ and $\eps$ such that if $p_n \geq Cn^{-2/3}$, then a.a.s.~every subhypergragraph of $G^{(3)}(n,p_n)$ with at least $\left(\frac{3}{4} -\eps\right)\binom{n}{3}p_n$ edges that does not contain the~Fano plane may be made bipartite by removing from it at most $\delta n^3p_n$ edges.
\end{thm}

Our methods imply analogous statements for the~other two hypergraphs mentioned above. These statements, Theorem~\ref{thm:F5F7-stability-Gnp} below, can be deduced from the~arguments used in~\cite{CoGo} under the~somewhat stronger assumption that $p_n \geq n^{-1}(\log n)^C$.

\begin{thm}
  \label{thm:F5F7-stability-Gnp}
  For every positive $\delta$, there exist positive constants $C$ and $\eps$ such that if $p_n \geq Cn^{-1}$, then a.a.s.~the following holds:
  \begin{enumerate}[(i)]
  \item
    \label{item:F5-Gnp}
    For every subhypergraph of $G^{(3)}(n,p_n)$ with at least $\left(\frac{2}{9} - \eps\right)\binom{n}{3}p_n$ edges that does not contain the~$3$-book of $2$ pages, there exists a~partition of $[n]$ into sets $V_1$, $V_2$, and $V_3$ such that all but at most $\delta n^3 p_n$ edges have one point in each $V_i$.
    \item
      \label{item:F7-Gnp}
    For every subhypergraph of $G^{(4)}(n,p_n)$ with at least $\left(\frac{3}{8} - \eps\right)\binom{n}{4}p_n$ edges that does not contain the~$4$-book of $3$ pages, there exists a~partition of $[n]$ into sets $V_1$ and $V_2$ such that all but at most $\delta n^4 p_n$ edges have two points in each $V_i$.      
  \end{enumerate}
\end{thm}

We remark that, similarly as in Theorem~\ref{thm:stability-Gnp-new}, the~``asymptotically almost surely'' in the~statement of Theorem~\ref{thm:F5F7-stability-Gnp} can be replaced by ``with probability at least $1 - \exp(-\Omega(n^3p_n))$'' in (\ref{item:F5-Gnp}) and ``with probability at least $1 - \exp(-\Omega(n^4p_n))$'' in (\ref{item:F7-Gnp}).

\subsubsection{Sum-free sets}

\label{sec:new-sum-free-sets}

A~\emph{Schur triple} in an~Abelian group $G$ is any triple $(x,y,z) \in G^3$ that satisfies the~equation $x + y = z$. A~set $A \subseteq G$ is called \emph{sum-free} if $A^3$ contains no Schur triples or, in other words, if $(A+A) \cap A = \emptyset$. Here is an~important definition in the~study of sum-free sets: We say that a~finite Abelian group $G$ is of type~I if $|G|$ has a~prime divisor $q$ with $q \equiv 2 \pmod 3$ and it is of type~I($q$) if $q$ is the~smallest such prime. For a~set $B \subseteq G$, let $\mu(B)$ be the~density of the~largest sum-free subset of $B$ (so that this subset has $\mu(B)|B|$ elements). Diananda and Yap~\cite{DiYa} showed that $\mu(G) = \frac{1}{3} + \frac{1}{3q}$ for~all $G$ of type~I($q$) and characterised all sum-free subsets of $G$ with $\mu(G)|G|$ elements. The~results of Conlon and Gowers and Schacht yield the~following statement.

\begin{thm}[{\cite{CoGo,Sc}}]
  \label{thm:sum-free-density}
  Let $q$ be a~prime with $q \equiv 2 \pmod 3$ and let $(G_n)$ be a~sequence of type~I($q$) groups satisfying $|G_n| = n$. Then for every positive $\eps$, there exist positive constants $c$ and $C$ such that
  \[
  \lim_{n \to \infty} P\left( \mu((G_n)_{p_n}) \leq \left(\frac{1}{3} + \frac{1}{3q} + \eps\right) n p_n\right) =
  \begin{cases}
    1, & \text{if $p_n \geq Cn^{-1/2}$},\\
    0, & \text{if $p_n \leq cn^{-1/2}$}.
  \end{cases}
  \]
\end{thm}

It was proved by Green and Ruzsa~\cite{GrRu} that the~property of being sum-free in a~group of type~I exhibits very strong stability.

\begin{thm}[{\cite{GrRu}}]
  \label{thm:stability-groups}
  Let $G$ be an~Abelian group of type $I(q)$. If $A$ is a~sum-free subset of $G$ and
  \[
  |A| \ge \left( \mu(G) - \frac{1}{3q^2+3q} \right)|G|,
  \]
  then $A$ is contained in some sum-free set $A'$ of maximum size.
\end{thm}

As a~last application of our main result, we will give a~much more transparent proof of the~following sparse random analogue of Theorem~\ref{thm:stability-groups}, originally derived from the~transference theorem of Conlon and Gowers~\cite{CoGo} by Balogh, Morris, and Samotij~\cite{BaMoSa}, with an~improved probability estimate.

\begin{thm}[{\cite{BaMoSa, CoGo}}]
  \label{thm:stability-groups-random}
  Let $q$ be a~prime with $q \equiv 2 \pmod 3$ and let $(G_n)$ be a~sequence of type~I($q$) groups satisfying $|G_n| = n$. Then for every positive $\delta$, there exist positive constants $\eps$ and $C$ such that with probability at least $1 - \exp(-\Omega(np_n))$, for every sum-free subset $A \subseteq (G_n)_{p_n}$ with at least $(\mu(G_n) - \eps)np_n$ elements, there exists a~sum-free set $A' \subseteq G_n$ of maximum size such that $|A \setminus A'| \leq \delta n p_n$.
\end{thm}

\subsection{Notation}

\label{sec:notation}

Given a~(hyper)graph $H$, we denote its vertex and edge sets by $V(H)$ and $E(H)$, and the~cardinalities of these two sets by $v(H)$ and $e(H)$, respectively. As one often identifies the~hypergraph $H$ with its edge set $E(H)$, sometimes instead of $e(H)$ or $|E(H)|$, we will simply write $|H|$. For a~set $U \subseteq V(H)$, we write $H[U]$ to denote the~subhypergraph of $H$ induced by $U$, i.e., the~hypergraph on the~vertex set $U$ whose edges are all the~edges of $H$ that are fully contained in $U$. Given a~vertex $v \in V(H)$, we let $\deg(v,U)$ denote the~degree of $v$ in $H[U]$, i.e., the~number of edges of $H[U]$ that contain $v$. Finally, we will denote by $[n]$ the~set $\{1, \ldots, n\}$ of the~first $n$ positive integers.

Since throughout the~paper, we will deal with many sequences indexed by (subsets of) the~natural numbers, in order to unclutter the~notation and, hopefully, improve readability, we use the~(somewhat informal) notational convention that the~sequences are denoted by boldface letters, e.g., $\bp$ stands for $(p_n)$, that is, the~sequence $p \colon \N \to [0,1]$ indexed by the~set of natural numbers\footnote{Since in order to aid readability, in the~proofs of our theorems we will often drop the~subscript $n$, we need a~way to distinguish between the~$n$th element of the~sequence, abbreviated by $p$, and the~sequence $p$ itself.}. the~only exception is that, due to typesetting limitations, the~sequence $(\B_n)$ will be denoted by $\bB$.

\subsection{Outline}

\label{sec:outline}

The remainder of this paper is organised as follows. In Section~\ref{sec:preliminaries}, we state a~few auxiliary results that will be used in our proofs. In Section~\ref{sec:main-results}, we state the~main result of this paper, Theorem~\ref{thm:main}, which we then prove in Section~\ref{sec:proof-lemma-main}. Finally, in Section~\ref{sec:proofs-new-results}, we use Theorem~\ref{thm:main} to deduce Theorems~\ref{thm:stability-Gnp-new}, \ref{thm:F5F7-stability-Gnp} and~\ref{thm:stability-groups-random}.

\section{Preliminaries}

\label{sec:preliminaries}

\subsection{Bounding large deviations}

\label{sec:bound-large-dev}

In the~proof of Lemma~\ref{lemma:main}, we will often use the~following standard estimates for tail probabilities of the~binomial distribution, see, e.g., \cite[Appendix A]{AlSp}.

\begin{lemma}[Chernoff's inequality]
  \label{lemma:Chernoff}
  Let $n$ be a~positive integer, let $p \in [0,1]$ and let $X \sim \Bin(n,p)$. For every positive $a$,
  \[
  P(X < np - a) < \exp\left(-\frac{a^2}{2np}\right) \quad \text{and} \quad P(X > np + a) < \exp\left(-\frac{a^2}{2np} + \frac{a^3}{2(np)^2}\right)
  \]
  In particular, if $a \leq np/2$, then
  \[
  P(X > np + a) < \exp\left(-\frac{a^2}{4np}\right).
  \]
\end{lemma}

We will also need the~following approximate concentration result for $(K,\bp)$-bounded hypergraphs. the~definition of $(K,\bp)$-boundedness and $\deg_i$ are given in Section~\ref{sec:main-results}, in Definition~\ref{dfn:boundedness} and in~\eqref{eq:deg-i}, respectively.

\begin{lemma}[{\cite{RoRu,Sc}}]
  \label{lemma:upper-tail}
  Let $\bp$ be a~sequence of probabilities, let $K$ be a~positive constant, and suppose that $\bH$ is a~sequence of $k$-uniform hypergraphs that is $(K,\bp)$-bounded and satisfies $|V(H_n)| \to \infty$ as $n \to \infty$. Then for every $i \in \{0, \ldots, k-1\}$ and every positive $\eta$, there exist positive $b$ and $N$ such that for~every $n$ with $n \geq N$ and every $q \in [0,1]$ with $q \geq p_n$, with probability at least $1 - \exp(-bq|V(H_n)|)$, there exists a~set $X \subseteq V(H_n)_q$ with $|X| \leq \eta q |V(H_n)|$ such that
  \[
  \sum_{v \in V(H_n)} \deg_i^2(v, V(H_n)_q \setminus X) \leq 4^kk^2Kq^{2i} \frac{|H_n|^2}{|V(H_n)|}.
  \]
\end{lemma}

We remark that~\cite{RoRu,Sc} stated Lemma~\ref{lemma:upper-tail} with the~assumption $i \in \{1, \ldots, k-1\}$; in the~case $i = 0$, the~assertion of the~lemma holds trivially (one can take $X = \emptyset$) with probability $1$.

\subsection{Stability theorems and removal lemmas}

\label{sec:stab-theor-remov}

The proof of Theorem~\ref{thm:stability-Gnp-new} will rely on the~following classical result known as the~\emph{graph removal lemma}, originally proved in the~case $H = K_3$ by Ruzsa and Szemer{\'e}di~\cite{RuSz}.

\begin{thm}
  \label{thm:graph-removal-lemma}
  For an~arbitrary graph $H$ and any positive constant $\delta$, there exists a~positive constant~$\eps$ such that every graph on $n$ vertices with at most $\eps n^{v(H)}$ copies of $H$ can be made $H$-free by removing from it at most $\delta n^2$ edges.
\end{thm}

The proof of Theorem~\ref{thm:F5F7-stability-Gnp} will rely on two aforementioned stability results for the~book hypergraphs and a~version of Theorem~\ref{thm:graph-removal-lemma} for hypergraphs.

\begin{thm}[\cite{FuPiSi,KeMu}]
  \label{thm:F5F7-stability}
  For every positive constant $\delta$, there exists a~positive constant $\eps$ such that the~following holds:
  \begin{enumerate}[(i)]
  \item
    \label{item:F5}
    For every $3$-uniform hypergraph with at least $\left(\frac{2}{9} - \eps\right)\binom{n}{3}$ edges that does not contain the~$3$-book of $2$ pages, there exists a~partition of $[n]$ into sets $V_1$, $V_2$, and $V_3$ such that all but at most $\delta n^3$ edges have one point in each $V_i$.
    \item
      \label{item:F7}
    For every $4$-uniform hypergraph with at least $\left(\frac{3}{8} - \eps\right)\binom{n}{4}$ edges that does not contain the~$4$-book of $3$ pages, there exists a~partition of $[n]$ into sets $V_1$ and $V_2$ such that all but at most $\delta n^4$ edges have two points in each $V_i$.      
 \end{enumerate}
\end{thm}

\begin{thm}[{\cite{Go,NaRoSc,Ta}}]
  \label{thm:hypergraph-removal-lemma}
  For an~arbitrary $k$-uniform hypergraph $H$ and any positive constant $\delta$, there exists a~positive constant $\eps$ such that every $k$-uniform hypergraph on $n$ vertices with at most $\eps n^{v(H)}$ copies of $H$ may be made $H$-free by removing from it at most $\delta n^k$ edges.
\end{thm}

Finally, the~proof of Theorem~\ref{thm:stability-groups-random} will use the~following corollary of Theorem~\ref{thm:stability-groups} and the~so-called \emph{removal lemma for Abelian groups} proved by Green~\cite{Gr}.

\begin{cor}[{\cite{BaMoSa}}]
  \label{cor:group-stability}
  Let $q$ be a~prime with $q \equiv 2 \pmod 3$, let $G$ be a~group of type $I(q)$, and let $\eps$ be a~constant satisfying $0 < \eps < 1/(9q^2+9q)$. Then every $A \subseteq G$ with $|A| \ge (\mu(G) - \eps) |G|$ either contains at least $\eps^3|G|^2/27$ Schur triples or satisfies $|A \setminus A'| \le \eps|G|$ for some sum-free set $A'$ of~maximum size.
\end{cor}

\section{Main results}

\label{sec:main-results}

Following~\cite{Sc}, we will phrase the~main result in the~language of sequences $\bH$ of uniform hypergraphs. In the~setting of Theorem~\ref{thm:stability-Gnp-new}, we have $V(H_n) = E(K_n)$ and the~edges of $H_n$ are edge sets of copies of a~fixed graph $H$ in $K_n$. Similarly, in Theorem~\ref{thm:F5F7-stability-Gnp}, $H_n$ represents copies of~the~appropriate book hypergraph in the~complete $3$- or $4$-uniform hypegraph on $n$ vertices. In~the~setting of Theorem~\ref{thm:stability-groups-random}, the~vertex set of $H_n$ will be the~set of elements of some Abelian group $G_n$ of order $n$, whereas the~edges of $H_n$ will be triples $\{x,y,z\}$ satisfying $x + y = z$. Since we are heavily borrowing from the~paper of Schacht~\cite{Sc}, our presentation and notation closely follow~\cite{Sc}.

In order to transfer an~extremal result from the~deterministic to the~probabilistic setting, both the~result of Conlon and Gowers~\cite{CoGo} and the~one of Schacht~\cite{Sc} require a~more robust version of this extremal result. One needs to assume that every sufficiently dense substructure (e.g., sufficiently large subgraph of the~complete graph) not only contains one copy of the~forbidden configuration (e.g., a~copy of a~fixed graph $H$), but also that the~number of copies of the~forbidden configuration in this substructure is of the~same order of magnitude as the~total number of copies of this configuration in the~full structure. Note that in many natural settings, such property does hold (e.g., by the~supersaturation theorem of Erd{\H o}s and Simonovits~\cite{ErSi83}). the~following definition makes this condition rigorous.

\begin{dfn}[{\cite{Sc}}]
  Let $\bH$ be a~sequence of $k$-uniform hypergraphs and let $\alpha$ be a~nonnegative real. We say that $\bH$ is \emph{$\alpha$-dense} if for every positive $\delta$, there exist positive $\eps$ and $N$ such that for every $n$ with $n \geq N$ and every $U \subseteq V(H_n)$ with $|U| \geq (\alpha + \delta)|V(H_n)|$, we have $|H_n[U]| \geq \eps |H_n|$.
\end{dfn}

Similarly as in~\cite{CoGo}, in order to transfer a~stability result from the~deterministic to the~probabilistic setting, we will need a~robust version of this stability result. Here, we need to assume that every sufficiently dense substructure is either close to some special substructure (e.g., a~$(\chi(H)-1)$-partite graph) or it contains many copies of the~forbidden configuration. Again, note that in many natural settings, such property does hold (e.g., as a~consequence of the~Erd{\H o}s-Simonovits stability theorem, Theorem~\ref{thm:stability}, and the~removal lemma for graphs, Theorem~\ref{thm:graph-removal-lemma}; see the~proof of Theorem~\ref{thm:stability-Gnp-new} in Section~\ref{sec:proofs-new-results}). the~following definition makes this condition rigorous.

\begin{dfn}[{\cite{AlBaMoSa}}]
  \label{dfn:stability}
  Let $\bH$ be a~sequence of $k$-uniform hypegraphs, let $\alpha$ be a~positive real and let $\bB$ be a~sequence of sets with $\B_n \subseteq \P(V(H_n))$. We say that $\bH$ is \emph{$(\alpha, \bB)$-stable} if for every positive~$\delta$, there exist positive $\eps$ and $N$ such that for every $n$ with $n \geq N$ and every $U \subseteq V(H_n)$ with $|U| \geq (\alpha - \eps)|V(H_n)|$, we have either $|H_n[U]| \geq \eps |H_n|$ or $|U \setminus B| \leq \delta |V(H_n)|$ for some $B \in \B_n$.
\end{dfn}

The second condition in Theorem~\ref{thm:main}, which one may view as a~measure of uniformity of the~distribution of copies of the~forbidden configuration in the~full structure, imposes a~lower bound on the~probability for which we can transfer our stability (extremal) result to the~random setting. Before we state this condition (Definition~\ref{dfn:boundedness}), we need to introduce some notation. For a~hypergraph $H$, a~vertex $v \in V(H)$, and a~set $U \subseteq V(H)$, let $\deg_i(v,U)$ denote the~number of edges of $H$ containing $v$ and at least $i$ vertices in $U \setminus \{v\}$. More precisely, let
\begin{equation}
  \label{eq:deg-i}
  \deg_i(v,U) = |\{ e \in H \colon v \in e \text{ and } |e \cap (U \setminus \{v\})| \geq i \}|.  
\end{equation}
For $q \in [0,1]$, we let $\mu_i(H,q)$ denote the~expected value of the~sum of squares of such degrees over all $v \in V(H)$ with $U$ replaced by the~$q$-random subset of $V(H)$, namely,
\[
\mu_i(H,q) = \Ex \left[ \sum_{v \in V} \deg_i^2(v, V_q) \right],
\]
where $V = V(H)$.

\begin{dfn}[{\cite{Sc}}]
  \label{dfn:boundedness}
  Let $\bH$ be a~sequence of $k$-uniform hypergraphs, let $\bp$ be a~sequence of probabilities, and let $K$ be a~positive constant. We say that $\bH$ is \emph{$(K,\bp)$-bounded} if for every $i \in \{0, \ldots, k-1\}$, there exists an~$N$ such that for every $n$ with $n \geq N$ and every $q \in [0,1]$ with $q \geq p_n$, we have
  \[
  \mu_i(H_n,q) \leq Kq^{2i}\frac{|H_n|^2}{|V(H_n)|}.
  \]
\end{dfn}

Finally, we are ready to state our main result, a~stability version of~\cite[Theorem~3.3]{Sc}.

\begin{thm}
  \label{thm:main}
  Let $\bH$ be a~sequence of $k$-uniform hypergraphs, let $\alpha$ be a~positive real, let $\bB$ be a~sequence of~sets with $\B_n \subseteq \P(V(H_n))$, and suppose that $\bH$ is $(\alpha, \bB)$-stable. Furthermore, let $K$ be a~positive real and let $\bp$ be a~sequence of probabilities such that $p_n^k |H_n| \to \infty$ as $n \to \infty$, $\bH$ is $(K,\bp)$-bounded, and $|\B_n| = \exp(o(p_n|V(H_n)|))$. Then for every positive $\delta$, there exist positive $\xi$, $b$, $C$, and $N$ such that for every $n$ with $n \geq N$ and every $q$ satisfying $C p_n \leq q \leq 1$, the~following holds with~probability at least $1 - \exp(-bq|V(H_n)|)$: Every subset $W \subseteq V(H_n)_q$ with $|W| \geq (\alpha-\xi) q |V(H_n)|$ that satisfies $|W \setminus B| \geq \delta q |V(H_n)|$ for every $B \in \B_n$ satisfies $|H[W]| \geq \xi q^k |H_n| > 0$.  
\end{thm}

\begin{remark}
  Note that unlike~\cite[Theorem~3.3]{Sc}, the~statement of Theorem~\ref{thm:main} no longer contains the~somewhat artificial assumption that $q \leq 1/\omega_n$ for some sequence $\bomega$ satisfying $\omega_n \to 0$ as~$n \to \infty$. This is due to our refined treatment of multiple exposure in the~proof of Lemma~\ref{lemma:main}, see Section~\ref{sec:mult-expos-trick} and the~discussion at the~beginning of Section~\ref{sec:proof-lemma-main}.
\end{remark}

Similarly as in~\cite{Sc}, Theorem~\ref{thm:main} will be derived from a~stronger statement, Lemma~\ref{lemma:main} below, which will be proved by induction. Before we state it, we need a~few more definitions. For a~$k$-uniform hypergraph $H$, sets $W$ and $U$ with $W \subseteq U \subseteq V(H)$, and an~integer $i \in \{0, \ldots, k\}$, we let $E_U^i$ denote the~edges of $H[U]$ that have at least $i$ vertices in $W$, namely,
\[
E_U^i(W) = \{e \in H[U] \colon |e \cap W| \geq i\}.
\]
Note that for every $U \subseteq V(H)$ and every $W \subseteq U$,
\[
E_U^0(W) = H[U] \quad\text{and}\quad E_U^k(W) = H[W].
\]

\begin{lemma}
  \label{lemma:main}
  Let $\bH$ be a~sequence of $k$-uniform hypergraphs, let $\alpha$ be a~positive real, let $\bB$ be a~sequence of sets with $\B_n \subseteq \P(V(H_n))$, and suppose that $\bH$ is $(\alpha, \bB)$-stable. Furthermore, let $K$ be a~positive real and let $\bp$ be a~sequence of probabilities such that $p_n^k |H_n| \to \infty$ as $n \to \infty$, $|\B_n| = \exp(o(p_n|V(H_n)|))$, and $\bH$ is $(K,\bp)$-bounded. Then for every $i \in \{0, \ldots, k\}$ and every positive $\delta$, there exist positive $\xi$, $b$, $C$, and $N$ such that for all $\beta, \gamma \in (0,1]$ with $\beta \gamma \geq \alpha - \xi$, every $n$ with $n \geq N$, and every $q$ satisfying $C p_n \leq q \leq 1$, the~following holds:\\
  For every $U \subseteq V(H_n)$ with $|U| \geq \beta|V(H_n)|$, with probability at least $1 - \exp(-bq|V(H_n)|)$, the~random set $U_q$ has the~following property: Every subset $W \subseteq U_q$ with $|W| \geq \gamma q |U|$ that satisfies $|W \setminus B| \geq \delta q |V(H_n)|$ for every $B \in \B_n$ satisfies $|E_U^i(W)| \geq \xi q^i |H_n|$.
\end{lemma}

\section{Proof of Lemma~\ref{lemma:main}}

\label{sec:proof-lemma-main}

The proof of Lemma~\ref{lemma:main} follows very closely the~proof of~\cite[Lemma~3.4]{Sc}. For easier comparison, our notation mirrors (with few minor changes) the~notation used in~\cite{Sc}. the~proof goes by induction on~$i$. Similarly as in~\cite{Sc}, the~base of the~induction (Section~\ref{sec:induction-base}), which can be viewed as a~justification of our choice of the~definition of $(\alpha,\bB)$-stability (Definition~\ref{dfn:stability}), follows very easily. the~proof of the~induction step (Section~\ref{sec:induction-step}) is much more involved. As in~\cite{Sc}, we construct the~elements of $E_U^{i+1}(W)$ in stages and hence we expose the~random set $U_q$ in several rounds, letting $U_q = U_{q_1} \cup \ldots \cup U_{q_R}$ for appropriately chosen $R$ and $q_1, \ldots, q_R$. Here comes the~main new obstacle. Unlike the~extremal setting considered in~\cite{Sc}, the~most important property of the~sets $W \subseteq U_q$ that we have to consider, i.e., $|W \setminus B| \geq \delta q |V(H_n)|$ for every $B \in \B_n$, no longer implies the~corresponding property, $|(W \cap U_{q_s}) \setminus B| \geq \delta' q_s |V(H_n)|$ for every $B \in \B_n$, in the~sets $U_{q_s}$. the~solution to this problem (Section~\ref{sec:mult-expos-trick}), which is the~main novelty in our approach, is analysing in more detail the~relations between the~probability space of the~random sets $U_q$ and the~richer space of the~sequences of random sets $U_{q_1}, \ldots, U_{q_R}$. Even though the~crucial property of the~set $W$ mentioned above does not imply the~analogous property relative to the~sets $U_{q_s}$ in every sequence $U_{q_1}, \ldots, U_{q_R}$, this does happen in a~\emph{typical} representation of $U_q$ as $U_{q_1} \cup \ldots \cup U_{q_R}$, see Claim~\ref{claim:star}. This observation allows us to replace our setting of a~single random set $U_q$ to the~setting of sequences of independent random sets $U_{q_1}, \ldots, U_{q_R}$, which, as already proved by~\cite{Sc}, is much more convenient to work in. Moreover, the~more rigorous treatment of the~equivalence between these two settings allow us to remove the~somewhat artificial assumption $q \leq 1/\omega_n$ that was necessary in the~approach taken in~\cite{Sc}. the~rest is as in the~proof of~\cite[Lemma~3.4]{Sc}. In each of the~$R$ rounds, we either construct `many' elements of $E_U^{i+1}(W)$ or, appealing to the~inductive assumption, we exhibit more than $\frac{1}{R} |V(H_n)|$ new `rich' vertices in $U$ that complete `many' elements of $E_U^i(W)$ to elements of $E_U^{i+1}(W)$, see Claim~\ref{claim:main}. Since the~latter can happen at most $R-1$ times ($U$ contains at most $|V(H_n)|$ vertices), Lemma~\ref{lemma:main} will follow.

\subsection{Setup}

Let $\bH$ be a~sequence of $k$-uniform hypergraphs, let $\bB$ be a~sequence of sets with $\B_n \subseteq \P(V(H_n))$, let $\bp$ be a~sequence of probabilities, and let $\alpha$ and $K$ be positive constants such that $\bH$ is $(\alpha,\bB)$-stable and $(K, \bp)$-bounded, $|\B_n| = \exp(o(p_n|V(H_n)|))$, and $p_n^k |H_n| \to \infty$ as $n \to \infty$. Note for future reference that since trivially $|H_n| \leq |V(H_n)|^k$, the~last assumption implies that
\begin{equation}
  \label{eq:pnVHn}
  p_n|V(H_n)| \to \infty \text{ as $n \to \infty$}.
\end{equation}
Finally, let $\delta$ be a~positive constant. We prove Lemma~\ref{lemma:main} by induction on $i$.

\subsection{Induction base $(i=0)$}

\label{sec:induction-base}

The base of induction follows quite easily from the~$(\alpha, \bB)$-stability of $\bH$. Let $\xi = \eps_{\ref{dfn:stability}}(\delta/2)$ and assume that $n \geq N$, where $N$ is sufficiently large; in particular, $N \geq N_{\ref{dfn:stability}}(\delta/2)$. Moreover, let $b = \delta/32$ and $C=1$. For the~sake of clarity of the~presentation, let $H = H_n$, let $V = V(H_n)$, and let $\B = \B_n$. Let $\beta, \gamma \in (0,1]$ satisfy $\beta\gamma \geq \alpha - \xi$, let $q$ satisfy $q \geq Cp_n$, and fix some $U \subseteq V$ with $|U| \geq \beta |V|$. Since $H$ is $(\alpha, \bB)$-stable and $|U| \geq (\alpha - \xi)|V|$, if $|U \setminus B| > (\delta/2)|V|$ for every $B \in \B$, then
\[
|E_U^0(W)| = |H[U]| \geq \xi |H|,
\]
so we may assume that $|U \setminus B| \leq (\delta/2)|V|$ for some $B \in \B$. Observe that for every $W \subseteq U_q$, one clearly has $|W \setminus B| \leq |U_q \setminus B|$. Hence, by Chernoff's inequality, with probability at least $1 - \exp(-bq|V|)$, the~set $U_q$ (vacuously) has the~claimed property, as it satisfies $|U_q \setminus B| < \delta q |V|$.

\subsection{Induction step $(i \to i+1)$}

\label{sec:induction-step}

Let $\xi'$, $b'$, $C'$, and $N'$ be the~constants whose existence is asserted by the~inductive assumption with parameters $i$ and $\delta/4$, i.e., let
\[
\xi' = \xi_{\ref{lemma:main}}(i, \delta/4), \quad b' = b_{\ref{lemma:main}}(i, \delta/4), \quad C' = C_{\ref{lemma:main}}(i, \delta/4), \quad\text{and}\quad N' = N_{\ref{lemma:main}}(i, \delta/4).
\]
We also let $\eta = \min\{\xi'/4, \delta/8\}$ and let $\hat{b} = b_{\ref{lemma:upper-tail}}(\eta)$. Throughout the~proof, we will assume that $n \geq N$, where $N$ is sufficiently large; in particular, $N \geq \max\{ N', N_{\ref{lemma:upper-tail}}(\eta)\}$. Similarly as before, for the~sake of clarity of the~presentation, we let $H = H_n$, $V = V(H_n)$, $\B = \B_n$, and $p = p_n$. We first define some constants. We set
\begin{equation}
  \label{eq:R}
  R = \left\lceil \frac{4^{k+1}k^2K}{(\xi')^2} + 1 \right\rceil
\end{equation}
and let
\[
\xi = \frac{(\xi')^2}{8k(RL^R)^{i+1}},
\quad
b = \min\left\{ \frac{(\xi')^2}{16^2},  \frac{b^*}{40RL^R} \right\},
\quad\text{and}\quad
C = RL^RC',
\]
where
\[
b^* = \min \left\{ \frac{\beta(\xi')^2}{16^2}, \frac{b'}{4}, \frac{\hat{b}}{4} \right\}
\quad\text{and}\quad
L = \frac{3}{b^*}.
\]
Finally, let $\beta, \gamma \in (0,1]$ satisfy $\beta\gamma \geq \alpha - \xi$, let $q$ satisfy $q \geq Cp_n$, and fix some $U \subseteq V$ with $|U| \geq \beta |V|$. Note that WLOG we may assume that $|U| = \beta |V|$ and that $\xi' \leq \delta/2$.

\subsubsection{Multiple exposure trick}

\label{sec:mult-expos-trick}

Let $\S$ denote the~event that the~random set $U_q$ possesses the~postulated stability property:

\smallskip
\noindent
\begin{center}
  \begin{tabular}{rp{0.85\textwidth}}
    $\S$: & Every subset $W \subseteq U_q$ with $|W| \geq \gamma q |U|$ that satisfies $|W \setminus B| \geq \delta q |V|$ for every $B \in \B$ satisfies $|E_U^{i+1}(W)| \geq \xi q^{i+1} |H|$.
  \end{tabular}
\end{center}

\smallskip
\noindent
In order to estimate the~probability of $\S$, we will consider a~richer probability space that is in~a~natural correspondence with the~space $\P(U)$ of all subsets of $U$ equipped with the~obvious probability measure $P$, i.e., the~distribution of the~random variable $U_q$. To this end, let $q_1, \ldots, q_R \in [0,1]$ be the~unique sequence of numbers that satisfies
\begin{equation}
  \label{eq:qs}
  1 - q = \prod_{s=1}^R (1 - q_s) \quad\text{and}\quad q_{s+1} = L q_s \text{ for every $s \in [R-1]$},
\end{equation}
and observe that
\begin{equation}
  \label{eq:qs-bounds}
  \sum_{s = 1}^R q_s \geq q \quad\text{and consequently}\quad q_s \geq q_1 \geq \frac{q}{RL^R}\text{ for every $s \in [R]$}.
\end{equation}
The richer probability space will be the~space $\P(U)^R$ equipped with the~product measure $P^*$ that is the~distribution of the~sequence $(U_{q_1}, \ldots, U_{q_R})$ of independent random variables, where for each $s$, the~variable $U_{q_s}$ is a~$q_s$-random subset of $U$. Crucially, observe that due to our choice of $q_1, \ldots, q_s$, see~\eqref{eq:qs}, the~natural mapping
\[
\varphi \colon \P(U)^R \to \P(U) \quad \text{defined by} \quad \varphi(U_1, \ldots, U_R) = U_1 \cup \ldots \cup U_R
\]
is measure preserving, i.e., for every $U_0 \subseteq U$,
\[
P(U_0) = P^*(\varphi^{-1}(U_0)).
\]
In other words, the~variables $U_q$ and $U_{q_1} \cup \ldots \cup U_{q_R}$ have the~same distribution. Finally, let $\delta^* = \delta/2$, let $\gamma^* = \gamma - \xi'/4$, and consider the~following event in the~space $\P(U)^R$:

\smallskip
\noindent
\begin{center}
  \begin{tabular}{rp{0.85\textwidth}}
    $\S^*$: & For every $W_1 \subseteq U_{q_1}, \ldots, W_R \subseteq U_{q_R}$ such that $|W_s| \geq \gamma^* q_s |U|$ and $|W_s \setminus B| \geq \delta^* q_s |V|$ for every $s \in [R]$ and every $B \in \B$, we have $|E_U^{i+1}(W_1 \cup \ldots \cup W_R)| \geq \xi q^{i+1} |H|$.
  \end{tabular}
\end{center}

\smallskip
\noindent
There are two reasons why we consider the~probability space $\P(U)^R$. the~first reason is that the~probability of $\S^*$ is much easier to estimate than the~probability of $\S$. the~second reason is that a~lower bound on $P^*(\S^*)$ implies a~(marginally weaker) lower bound on $P(\S)$, which we show below.

\begin{claim}
  \label{claim:star}
  $1 - P(S) \leq 2 \cdot (1 - P^*(S^*))$.
\end{claim}
\begin{proof}
  Note that in order to prove the~claim, it suffices to show that
  \[
  P^*(\S^* \mid \varphi^{-1}(\bU)) = P^*(\S^* \mid U_{q_1} \cup \ldots \cup U_{q_R} = \bU) \leq 1/2
  \]
  for every $\bU$ that does not satisfy $\S$. Consider an~arbitrary $\bU \subseteq U$ that does not satisfy $\S$. By~the~definition of $\S$, there exists a~set $W \subseteq \bU$ with $|W| \geq \gamma q |U|$ that satisfies $|W \setminus B| \geq \delta q |V|$ for every $B \in \B$ and $|E_U^{i+1}(W)| < \xi q^{i+1} |H|$. Consider the~event $\varphi^{-1}(\bU)$, i.e., the~event $U_{q_1} \cup \ldots \cup U_{q_R} = \bU$. Now, for each $s \in [R]$, let $W_s = W \cap U_{q_s}$. Since clearly $E_U^{i+1}(W) = E_U^{i+1}(W_1 \cup \ldots\cup W_R)$, it suffices to show that with probability at least $1/2$, we have $|W_s| \geq \gamma^*q_s|U|$ and $|W_s \setminus B| \geq \delta^*q_s|V|$ for every $s \in [R]$ and $B \in \B$.

  To this end, observe that conditioned on the~event $U_{q_1} \cup \ldots \cup U_{q_R} = \bU$, for each $s \in [R]$, the~variable $U_{q_s}$ has the~same distribution as $\bU_{q'_s}$, where $q'_s = q_s / q$ (although $U_{q_1}, \ldots, U_{q_R}$ are no longer independent). Recalling the~definitions of $\gamma^*$ and $\delta^*$, it now follows from Chernoff's inequality that for fixed $s \in [R]$ and $B \in \B$, both the~probability that $|W_s| < \gamma^* q_s |U| = \gamma^* q_s' q |U|$ and the~probability that $|W_s| < \delta^* q_s |V| = \delta^*q_s' q |V|$ are at most $\exp(-cq_s|V|)$, where $c$ is some positive constant depending only on $\alpha$, $\delta$, and $\xi'$. Since $q_s \geq q/(RL^R) \geq p$ for every $s \in [R]$, see~\eqref{eq:qs-bounds}, $|\B_n| = \exp(o(p_n|V(H_n)|))$, and \eqref{eq:pnVHn}, then the~claimed estimate follows from the~union bound, provided that $n$ is sufficiently large.
\end{proof}

\subsubsection{Estimating the~probability of $\S^*$}

In the~remainder of the~proof, we will work in the~space $\P(U)^R$ and estimate the~probability of the~event $S^*$, that is, $P^*(S^*)$. Let $U_{q_1}, \ldots, U_{q_R}$ be independent random subsets of $U$. Given $W_1 \subseteq U_{q_1}, \ldots, W_R \subseteq U_{q_R}$, let
\[
W(s) = (W_1, \ldots, W_s) \quad\text{and}\quad U(s) = (U_{q_1}, \ldots, U_{q_s}).
\]
We consider the~set $Z_s$ of `rich' vertices that extend many sets in $E_U^i(W_s)$ defined by
\[
Z_s = \left\{ u \in U \colon \deg_i(u, W_s, U) \geq \frac{\xi'}{2} q_s^i \frac{|H|}{|V|} \right\},
\]
where
\[
\deg_i(u, W_s, U) = \left| \left\{ e \in H \colon |e \cap (W_s \setminus \{u\})| \geq i \text{ and } e \subseteq U\right\} \right|,
\]
and let
\[
Z(s) = Z_1 \cup \ldots \cup Z_s.
\]
Now comes the~key step in the~proof. We show that with very high probability, for every $s \in [R]$, regardless of what happens in rounds $1, \ldots, s-1$, either $E_U^{i+1}(W_1 \cup \ldots \cup W_s)$ is large or the~set $Z(s)$ of `rich' elements grows by more than $|V|/R$.

\begin{claim}
  \label{claim:main}
  For every $s \in [R]$ and every choice of $W(s-1) \in \P(U)^{s-1}$, let $\S_{W(s-1)}^*$ denote the~event that $U_{q_s}$ has the~following property: For every $W_s \subseteq U_{q_s}$ with $|W_s| \geq \gamma^* q_s |U|$ that satisfies $|W_s \setminus B| \geq \delta^* q_s |V|$ for every $B \in \B_n$, either
  \begin{equation}
    \label{eq:done}
    |E_U^{i+1}(W_1 \cup \ldots \cup W_s)| \geq \xi q^{i+1} |H|
  \end{equation}
  or
  \begin{equation}
    \label{eq:Z-grows}
    |Z(s) \setminus Z(s-1)| \geq \frac{(\xi')^2}{4^{k+1}k^2K} |V|.
  \end{equation}
  Then for every $\bU \in \P(U)^{s-1}$,
  \[
  P^*(\S^*_{W(s-1)} \mid U(s-1) = \bU) \geq 1 - \exp(-2b^*q_s|V|),
  \]
  where $P^*(\S^*_{W(0)} \mid U(0) = \bU) = P^*(\S^*_{W(0)})$.
\end{claim}

\subsubsection{Deducing Lemma~\ref{lemma:main} from Claims~\ref{claim:star} and~\ref{claim:main}}

For every $t \in [R]$, let $\A_t$ denote the~event that $|U_{q_t}| \leq 2q_t|V|$ and let $\A(s) = \A_1 \cap \ldots \cap \A_s$. Observe that by~\eqref{eq:qs},
\begin{equation}
  \label{eq:Us-}
  \sum_{t=1}^{s-1} 2q_t|V| \leq 3q_{s-1}|V| \leq \frac{3q_s|V|}{L}.
\end{equation}
By Chernoff's inequality, \eqref{eq:pnVHn}, and \eqref{eq:qs-bounds},
\[
P^*(\neg \A(s-1)) \leq \sum_{t=1}^{s-1} \exp\left(-\frac{q_t|V|}{16}\right) \leq \exp\left(-\frac{q_1|V|}{20} \right).
\]
Now, for every $s \in [R]$, let $\S_s^*$ denote the~event that $\A(s-1)$ holds and $S_{W(s-1)}^*$ holds for all $W(s-1) \subseteq U(s-1)$\footnote{We write $W(s-1) \subseteq U(s-1)$ to denote the~fact that $W_t \subseteq U_t$ for all $t \in [s-1]$}. Observe that if $S_s^*$ holds for all $s \in [R]$, then $S^*$ must hold since~\eqref{eq:Z-grows} in~Claim~\ref{claim:main} can occur at most $R-1$ times, see~(\ref{eq:R}). Let
\[
\bUU = \left\{ \bU \in \P(U)^{s-1} \colon |\bU_t| \leq 2q_t|V| \text{ for all $t \in [s-1]$}\right\}
\]
and note that
\begin{align}
  \label{eq:P*Ss*}
  P^*(\neg \S_s^*) & \leq P^*(\neg \A(s-1)) + P^*(\neg \S_s^* \wedge \A(s-1)) \\
  \nonumber
  &  = P^*(\neg \A(s-1)) + \sum_{\bU \in \bUU} P^*(\neg \S_s^* \wedge U(s-1) = \bU).
\end{align}
Now, by Claim~\ref{claim:main}, for every $\bU \in \bUU$,
\begin{align}
  \label{eq:P*Ss*bU}
  P^*(\neg \S_s^* \wedge U(s-1) = \bU) & = P^*(\neg \S_s^* \mid U(s-1) = \bU) \cdot P^*(U(s-1) = \bU) \\
  \nonumber
  & \leq \sum_{W(s-1) \subseteq \bU} P^*(\neg \S_{W(s-1)}^* \mid U(s-1) = \bU) \cdot P^*(U(s-1) = \bU) \\
  \nonumber
  & \leq 2^{\sum_{t=1}^{s-1}2q_t|V|} \cdot \exp(-2b^*q_s|V|) \cdot P^*(U(s-1) = \bU).
\end{align}
Since clearly $\sum_{\bU \in \bUU} P^*(U(s-1) = \bU) \leq 1$, it follows from~\eqref{eq:Us-}, \eqref{eq:P*Ss*}, and \eqref{eq:P*Ss*bU} that
\begin{align}
  \label{eq:P*S*}
  P^*(\neg S^*) & \leq \sum_{s=1}^R P^*(\neg S_s^*) \leq R\exp\left( -\frac{q_1|V|}{20} \right) + \sum_{s=1}^R 2^{\frac{3q_s|V|}{L}}\exp(-2b^*q_s|V|) \\
  \nonumber
  & \leq R\exp\left(-\frac{q_1|V|}{20}\right) + R\exp(-b^*q_1|V|) \leq \frac{1}{2}\exp(-bq|V|).
\end{align}
Now, Lemma~\ref{lemma:main} easily follows from \eqref{eq:P*S*} and Claim~\ref{claim:star}.

\subsubsection{Proof of Claim~\ref{claim:main}}

Let $s \in [R]$, condition on the~event $U(s-1) = \bU$ for some $\bU \in \P(U)^{s-1}$ and assume that $W(s-1)$ is given. Note that this uniquely defines $Z(s-1)$. Also, observe that it follows from the~definition of $Z(s-1)$ and~\eqref{eq:qs-bounds} that
\begin{align}
  \label{eq:EUiWs}
  |E_U^{i+1}(W(s))| & \geq \frac{1}{k} \sum_{w \in W_s} \deg_i(w, W(s-1), U) \geq \frac{1}{k} \cdot |W_s \cap Z(s-1)| \cdot \frac{\xi'}{2}q_1^i\frac{|H|}{|V|} \\
  \nonumber
  & \geq \frac{|W_s \cap Z(s-1)|}{q_s} \cdot \frac{\xi'}{2}\frac{q^{i+1}}{k(RL^R)^{i+1}}\frac{|H|}{|V|} = \frac{|W_s \cap Z(s-1)|}{(\xi'/4) q_s |V|} \cdot \xi q^{i+1} |H|,
\end{align}
hence it will be enough if we show that
\begin{equation}
  \label{eq:Ws-Zs-}
  |W_s \cap Z(s-1)| \geq \frac{\xi'}{4} q_s |V|.
\end{equation}
We consider two cases, depending on the~cardinality of $Z(s-1)$.

\medskip
\noindent
{\bf Case 1. $|U \setminus Z(s-1)| < (\gamma^* - \xi'/2)|U|$.}

\smallskip
\noindent
By Chernoff's inequality, with probability at least $1 - \exp(-2b^*q_s|V|)$, the~set $U_{q_s}$ satisfies $|U_{q_s} \setminus Z(s-1)| \leq (\gamma^* - \xi'/4)q_s|U|$. Consequently, for every $W_s \subseteq U$ with $|W_s| \geq \gamma^* q_s |U|$, we have
\[
|W_s \cap Z(s-1)| \geq |W_s| - |U_{q_s} \setminus Z(s-1)| \geq \frac{\xi'}{4} q_s |U|,
\]
which, by~\eqref{eq:EUiWs}, proves~\eqref{eq:done}, see~\eqref{eq:Ws-Zs-}.

\medskip
\noindent
{\bf Case 2. $|U \setminus Z(s-1)| \geq (\gamma^* - \xi'/2)|U|$.}

\smallskip
\noindent
In this case, we will apply the~inductive assumption to the~set $U \setminus Z(s-1)$. First, observe that if $|W_s \cap Z(s-1)| \geq (\xi'/4)q_s|V|$, then this, by~\eqref{eq:EUiWs}, proves~\eqref{eq:done}, see~\eqref{eq:Ws-Zs-}. Hence, from now on we may assume that the~inverse inequality holds, i.e., that
\begin{equation}
  \label{eq:Ws-Zs-neg}
  |W_s \cap Z(s-1)| < \frac{\xi'}{4} q_s |V|.
\end{equation}
Let
\[
U' = U \setminus Z(s-1), \quad \beta' = \frac{|U'|}{|V|}, \quad \text{and} \quad \gamma' = \left(\gamma^* - \frac{\xi'}{2}\right) \frac{|U|}{|U'|}.
\]
Clearly, $\beta', \gamma' \in (0,1]$ and
\[
\beta'\gamma' = \left( \gamma^* - \frac{\xi'}{2} \right) \cdot \frac{|U|}{|V|} = \left( \gamma^* - \frac{\xi'}{2} \right)\beta = \left(\gamma - \frac{3}{4}\xi' \right)\beta \geq \beta\gamma - \frac{3}{4}\xi' \geq \alpha - \xi - \frac{3}{4}\xi' \geq \alpha - \xi'.
\]
Note that by~(\ref{eq:qs-bounds}) and our assumption on $q$ and $C$, we have $q_s \geq \frac{q}{RL^R} \geq \frac{Cp}{RL^R} \geq C'p$ and hence by~the~inductive assumption applied to $U'$, with probability at least $1 - \exp(-b'q_s|V|)$, every subset $W' \subseteq U_{q_s}'$ with $|W'| \geq \gamma'|U_{q_s}'|$ such that $|W' \setminus B| \geq (\delta/4)q_s|V|$ for every $B \in \B$ satisfies $|E_{U'}^{i}(W')| \geq \xi'q_s^i|H|$. Moreover, it follows from Lemma~\ref{lemma:upper-tail} that with probability at least $1 - \exp(-\hat{b}q_s|V|)$, there exists a~set $X \subseteq U'_{q_s}$ satisfying 
\begin{equation}
  \label{eq:X}
  |X| \leq \eta q_s |V| = \min\left\{ \frac{\xi'}{4}, \frac{\delta^*}{4} \right\}q_s|V|.
\end{equation}
and
\begin{equation}
  \label{eq:U'-L2-norm}
  \sum_{u \in U'} \deg_i^2(u, U_{q_s}' \setminus X) \leq 4^kk^2Kq_s^{2i}\frac{|H|^2}{|V|}.
\end{equation}
Consider the~set $W' \subseteq U'$ defined by $W' = W_s \setminus (X \cup Z(s-1))$. It follows from~\eqref{eq:Ws-Zs-neg} and~\eqref{eq:X} that
\[
|W'| \geq |W_s| - |W_s \cap Z(s-1)| - |X| \geq \left(\gamma^* - \frac{\xi'}{4} - \eta \right) q_s |V| \geq \gamma' q_s |V|.
\]
and that for every $B \in \B$,
\[
|W' \setminus B| \geq |W_s \setminus B| - |W_s \cap Z(s-1)| - |X| \geq \left(\delta^* - \frac{\xi'}{4} - \eta\right) q_s |V| \geq \frac{\delta}{4} q_s |V|.
\]
From the~inductive assumption (which, recall, holds for $U_{q_s}'$ with probability at least $1 - \exp(-b'q_s|V|)$), we infer that
\begin{equation}
  \label{eq:U'-L1-norm}
  \sum_{u \in U'} \deg_i(u, W', U') \geq |E_{U'}^i(W')| \geq \xi' q_s^i |H|.
\end{equation}
Let
\[
Z_s' = \left\{ u \in U' \colon \deg_i(u,W',U') \geq \frac{\xi'}{2}q_s^i\frac{|H|}{|V|} \right\}
\]
and note that, by definition, $Z_s' \subseteq Z_s$ and, by~\eqref{eq:U'-L1-norm},
\begin{equation}
  \label{eq:Zs'-L1-norm}
  \sum_{u \in Z_s'} \deg_i(u, W', U') \geq \sum_{u \in U'} \deg_i(u, W', U') - |U' \setminus Z_s'| \cdot \frac{\xi'}{2} q_s^i \frac{|H|}{|V|} \geq \frac{\xi'}{2} q_s^i |H|.      
\end{equation}
It follows form~\eqref{eq:U'-L2-norm}, \eqref{eq:Zs'-L1-norm}, and the~Cauchy-Schwarz inequality that
\begin{align*}
  4^kk^2Kq_s^{2i}\frac{|H|^2}{|V|} & \geq \sum_{u \in U'} \deg_i^2(u, U') \geq \sum_{u \in Z_s'} \deg_i^2(u, W', U') \\
  & \geq \frac{1}{|Z_s'|} \left( \sum_{u \in Z_s'} \deg_i(u,W',U') \right)^2 \geq \frac{1}{|Z_s'|} \left( \frac{\xi' q_s^i |H|}{2}\right)^2
\end{align*}
and consequently,
\[
|Z_s'| \geq \frac{(\xi')^2}{4^{k+1}k^2K}|V|.
\]
Since $Z_s' \subseteq U' = U \setminus Z(s-1)$, the~sets $Z_s'$ and $Z(s-1)$ are disjoint. Therefore, \eqref{eq:Z-grows} holds with probability at least
\[
1 - \exp(-b'q_s|V|) - \exp(-\hat{b}q|V|),
\]
which, by~\eqref{eq:pnVHn}, is at least $1 - \exp(-2b^*q_s|V|)$. This concludes the~proof of Claim~\ref{claim:main} and consequently, the~proof of Lemma~\ref{lemma:main}.

\section{Proofs of the~new results}

\label{sec:proofs-new-results}

In this section, we prove Theorems~\ref{thm:stability-Gnp-new} and~\ref{thm:stability-groups-random}. the~derivation of Theorem~\ref{thm:F5F7-stability-Gnp} from Theorems~\ref{thm:F5F7-stability}, \ref{thm:hypergraph-removal-lemma}, and~\ref{thm:main} and the~calculations done in~\cite{Sc} (proving that the~appropriate sequence of hypergraphs is $(K,\bp)$-bounded) does not differ much from the~proof of Theorem~\ref{thm:stability-Gnp-new} given below and hence we shall leave it to the~reader.

\begin{proof}[{Proof of Theorem~\ref{thm:stability-Gnp-new}}]
  Let $H$ be a~graph with at least one vertex contained in at least two edges. We want to apply Theorem~\ref{thm:main}. To this end, consider the~sequence $\bH$ of $e(H)$-uniform hypergraphs with $V(H_n) = E(K_n)$ and $E(H_n)$ consisting of edge sets of all copies of $H$ in $K_n$. Moreover, let $p_n = n^{-1/m_2(H)}$ and let $\B_n$ be the~family of edge sets of all complete $(\chi(H)-1)$-partite graphs on the~vertex set $[n]$. Observe that in order to complete the~proof, it suffices to verify that the~assumptions of Theorem~\ref{thm:main} are satisfied. Since $H$ contains a~vertex with degree at least $2$, we have that $m_2(H) \geq 1$ and hence
  \[
  p_n^{e(H)} |H_n| \geq p_n^{e(H)} \binom{n}{v(H)} = \Omega(p_n^{e(H)} n^{v(H)}) = \Omega(p_n n^2) = \Omega(n).
  \]
  Moreover, it was proved in~\cite{Sc} that the~sequence $\bH$ is $(K,\bp)$ bounded for some sufficiently large constant $K$. Finally, note that if $\chi(H) > 2$, then $H$ contains an~odd cycle (of length at most $v(H)$) and hence $m_2(H) \geq 1 + 1/(v(H)-2) > 1$. It follows that regardless of $\chi(H)$,
  \[
  |B_n| = (\chi(H) - 1)^n = \exp(o(p_n n^2)) = \exp(o(p_n|V(H_n)|)).
  \]

  Crucially, we need to verify that the~sequence $\bH$ is $\left( 1 - \frac{1}{\chi(H)-1}, \bB \right)$-stable. For that, we appeal to the~original stability theorem of Erd{\H o}s and Simonovits (Theorem~\ref{thm:stability}) and to the~graph removal lemma (Theorem~\ref{thm:graph-removal-lemma}). Fix a~positive $\delta$, let $\delta'' = \delta/5$, $\eps' = \eps_{\ref{thm:stability}}(\delta'')$, $\delta' = \min\{\delta'', \eps'/2\}$, and let $\eps = \min\{\eps'/2, \eps_{\ref{thm:graph-removal-lemma}}(\delta')\}$. Let $G$ be a~subgraph of $K_n$ with at least $\left(1 - \frac{1}{\chi(H)-1} - \eps\right)\binom{n}{2}$ edges that cannot be made $(\chi(H)-1)$-partite by removing from it $\delta\binom{n}{2}$ edges. We claim that it contains at least $\eps n^{v(H)}$ copies of $H$. If it did not, then by Theorem~\ref{thm:graph-removal-lemma}, removing at most $\delta' n^2$ edges from $G$ would make it into an~$H$-free graph $G'$. Since such $G'$ would still have at least $\ex(n,H) - (\eps + \delta')n^2$ edges, by Theorem~\ref{thm:stability} it could be made $(\chi(H)-1)$-partite by removing from it some further $\delta'' n^2$ edges. Hence, $G$ could be made bipartite by removing at most $2\delta''n^2$ edges, which is fewer than $\delta\binom{n}{2}$ edges, contradicting our assumption.
\end{proof}

\begin{proof}[{Proof of Theorem~\ref{thm:stability-groups-random}}]
  Let $q$ be a~prime with $q \equiv 2 \pmod 3$ and let $\bG$ be a~sequence of type $I(q)$ groups satisfying $|G_n| = n$. We want to apply Theorem~\ref{thm:main}. To this end, consider the~sequence $\bH$ of $3$-uniform hypergraphs with $V(H_n) = G_n$ and $E(H_n)$ consisting of all triples $\{x, y, z\}$ satisfying $x + y = z$ and note that $|H_n| = \Omega(n^2)$. Moreover, let $p_n = n^{-1/2}$ and let $\B_n$ be the~family of all maximum-size sum-free subsets of $G_n$. In order to complete the~proof, it suffices to show that the~assumptions of Theorem~\ref{thm:main} are satisfied. First, note that $p_n^3 |H_n| = \Omega(n^{1/2})$. Since for each Schur triple $\{x, y, z\} \in H_n$, there are only constantly many Schur triples $\{x',y',z'\} \in H_n$ intersecting $\{x, y, z\}$ in more than one element, an~easy computation (see~\cite{Sc}) shows that $\bH$ is $(K,\bp)$-bounded for sufficiently large constant $K$. Finally, since $|\B_n| \leq n$ (see, e.g.,~\cite[Corollary~3.4]{BaMoSa}), we have that $|\B_n| = \exp(o(n^{1/2})) = \exp(o(p_n|V(H_n)|)$. Crucially, we need to verify that $\bH$ is $\left(\frac{1}{3} + \frac{1}{3q}, \bB \right)$-stable. For that, we simply appeal to Corollary~\ref{cor:group-stability}.
\end{proof}

\bibliographystyle{amsplain}
\bibliography{stability}

\end{document}